\documentclass{amsart}
\usepackage{stmaryrd,mathabx}

\usepackage[pagebackref,
]{hyperref}  

 \newtheorem{thm}{Theorem}[section]
 \newtheorem{cor}[thm]{Corollary}
 \newtheorem{lem}[thm]{Lemma}
 \newtheorem{prop}[thm]{Proposition}
 \theoremstyle{definition}
 
 \theoremstyle{remark}

 \numberwithin{equation}{section}

\newcommand{\W}{\mathcal{W}}
\newcommand{\R}{\mathbb{R}}

\newcommand{\n}{\nabla}
\newcommand{\tg}{\widetilde{g}}
\newcommand{\tn}{\widetilde\nabla}
\newcommand{\tR}{\widetilde{R}}
\newcommand{\tF}{\widetilde{F}}
\newcommand{\tS}{\widetilde{S}}
\newcommand{\tQ}{\widetilde{Q}}
\newcommand{\tA}{\widetilde{A}}
\newcommand{\tP}{\widetilde{\Phi}}
\newcommand{\tN}{\widetilde{N}}
\newcommand{\ta}{\theta}
\newcommand{\lm}{\lambda}
\newcommand{\D}{\mathrm{d}}

\newcommand{\norm}[1]{\Vert#1\Vert ^2}
\newcommand{\nJ}{\norm{\nabla J}}
\newcommand{\tnJ}{\norm{\widetilde{\nabla} J}}

\newcommand{\propref}[1]{Proposition~\ref{#1}}
\newcommand{\lemref}[1]{Lemma~\ref{#1}}

\begin{document}
%
%
%
%
%
\title[Invariant Tensors on Almost Complex Norden Manifolds]
 {Invariant Tensors under the Twin Interchange\\
of Norden Metrics
on Almost Complex \\ Manifolds}
\author[M. Manev]{Mancho Manev}

\address{%
  Department of Algebra and Geometry \\
  Faculty of Mathematics and Informatics \\
  University of Plovdiv  \\
  236, Bulgaria Blvd. \\
  Plovdiv 4000, Bulgaria
}

\email{mmanev@uni-plovdiv.bg}




\begin{abstract}
The object of study are almost complex manifolds with a pair of Norden metrics, mutually associated by means of the almost complex structure. More precisely, a torsion-free connection and tensors with geometric interpretation are found which are invariant under the twin interchange, i.e. the swap of the counterparts of the pair of Norden metrics and the corresponding Levi-Civita connections. A Lie group depending on four real parameters is considered as an example of a 4-dimensional manifold of the studied type. The mentioned invariant objects are found in an explicit form.
\end{abstract}

\maketitle

\section*{Introduction}

Hermitian metrics on almost complex manifolds are well known. Then the almost complex structure $J$ acts as an isometry with respect to the (Riemannian or pseudo-Riemannian) metric. The associated (0,2)-tensor of the Hermitian metric is a 2-form. Other case is when the almost complex structure acts as an anti-isometry regarding a pseudo-Riemannian metric. Such a metric is called a Norden metric. The associated (0,2)-tensor of any Norden metric is also a Norden metric. So, in this case we dispose with a pair of mutually associated Norden metrics, known also as twin Norden metrics. These \emph{almost Norden manifolds} are studied in the latter three decades, in the beginning under the names generalized B-manifolds \cite{GrMeDj}, almost complex manifolds with Norden metric \cite{GaBo} and almost complex manifolds with B-metric \cite{GaGrMi}.

An interesting problem on almost Norden manifolds is the presence of tensors with some geometric interpretation which are invariant under the so-called twin interchange. This is the swap of the counterparts of the pair of Norden metrics and their Levi-Civita connections. Similar results for the considered manifolds in the basic classes $\W_1$ and $\W_3$ are obtained in \cite{MT06th} and \cite{DjDo}, \cite{MekManGri22}, respectively. The aim of the present work is the solving of the problem in the general case.

The present paper is organised as follows. Section~1 contains some preliminaries on the considered type of manifolds. In Section~2 we present the main results on the topic about the invariant objects and their vanishing. In Section~3 we consider an example of the studied manifolds of dimension 4 by means of a construction of an appropriate algebra depending on 4 real parameters. Then we compute the basic components of the invariant objects which are found in the previous section.

\section{Almost complex manifolds with Norden metrics}\label{sec_1}



Let $(M,J,g)$ be a $2n$-dimensional almost complex manifold with
Norden metrics or briefly an \emph{almost Norden manifold}.
This means that $J$ is an almost complex structure and $g$ is
a pseudo-Riemannian metric on $M$ such that
\begin{equation*}\label{str}
J^2x=-x,\quad g(Jx,Jy)=-g(x,y).
\end{equation*}
Here and further, $x$, $y$, $z$, $w$
will stand for arbitrary differentiable vector fields on $M$
or vectors in $T_pM$, $p\in M$.

On $(M,J,g)$, there exists an associated metric $\tg$ of $g$ given by
$\tg(x,y)=g(x,Jy)$. It is also a Norden
metric since $\tg(Jx,Jy)=-\tg(x,y)$ and the manifold $(M,J,\tg)$ is an almost Norden manifold, too. Both metrics are necessarily of neutral signature $(n,n)$. They are also known as \emph{twin Norden metrics} on $M$ because of
\begin{equation}\label{twin}
 \tg(x,y)=g(x,Jy),\quad \tg(x,Jy)=-g(x,y).
\end{equation}

The Levi-Civita connections of $g$ and $\tg$ are denoted by $\n$ and $\tn$, respectively.
The interchange of $\n$ and $\tn$ (and respectively $g$ and $\tg$) we call the \emph{twin interchange}.

The
tensor filed $F$ of type $(0,3)$ on $M$ is defined by
\begin{equation}\label{F}
F(x,y,z)=g\bigl( \left( \nabla_x J \right)y,z\bigr).
\end{equation}
It has the following properties
\begin{equation}\label{F-prop}
F(x,y,z)=F(x,z,y)=F(x,Jy,Jz).
\end{equation}

Let $\{e_i\}$ ($i=1,2,\dots,2n$) be an arbitrary basis of
$T_pM$ at a point $p$ of $M$. The components of the inverse matrix
of $g$ are denoted by $g^{ij}$ with respect to
$\{e_i\}$.

The Lee forms $\ta$ and $\ta^*$ associated with $F$ are defined by
\begin{equation}\label{ta}
\ta(z)=g^{ij}F(e_i,e_j,z),\quad \ta^*(z)=g^{ij}F(e_i,Je_j,z).
\end{equation}
For the 1-form $\ta^*$, using $\tg$, we have the following
\[
\ta^*(z)=g^{ij}F(e_i,Je_j,z)=J^j_kg^{ik}F(e_i,e_j,z)=-\tg^{ij}F(e_i,e_j,z)
\]
and the identity
\begin{equation}\label{ta*taJ}
\ta^*=-\ta\circ J
\end{equation}
holds by means of \eqref{F-prop}, because
\[
\ta^*(z)=g^{ij}F(e_i,Je_j,z)=-g^{ij}F(e_i,e_j,Jz)=-\ta(Jz).
\]

If the Levi-Civita connections of $g$ and $\tg$ are denoted by $\n$ and $\tn$, respectively, then the
following tensor is defined in \cite{GaGrMi} by
\begin{equation}\label{Phi}
    \Phi(x,y)=\tn_x y-\n_x y.
\end{equation}
This tensor is known also as the \emph{potential} of $\tn$ regarding $\n$ because of the formula
\begin{equation}\label{tn=nPhi}
    \tn_x y=\n_x y+\Phi(x,y).
\end{equation}
Since both the connections are torsion-free, $\Phi$ is symmetric, i.e. $\Phi(x,y)=\Phi(y,x)$.
Let the corresponding tensor of type $(0,3)$ with respect to $g$ be defined by
\begin{equation}\label{Phi03}
    \Phi(x,y,z)=g(\Phi(x,y),z).
\end{equation}
     By virtue of properties \eqref{F-prop} the following interrelations between $F$ and $\Phi$ are valid \cite{GaGrMi}
\begin{equation}\label{PhiF}
    \Phi(x,y,z)=\frac{1}{2}\bigl\{F(Jz,x,y)-F(x,y,Jz)-F(y,x,Jz)\bigr\},
\end{equation}
\begin{equation*}\label{FPhi}
    F(x,y,z)=\Phi(x,y,Jz)+\Phi(x,z,Jy).
\end{equation*}

A classification of the considered manifolds with respect to $F$
is given in \cite{GaBo}. All eight classes of almost Norden
manifolds are characterized there by
the properties of $F$. An equivalent classification in terms of $\Phi$ is proposed in \cite{GaGrMi}.
The three basic classes are defined, respectively:
\begin{equation}\label{class}
\begin{array}{l}
\W_1:\; F(x,y,z)=\frac{1}{2n} \bigl\{
g(x,y)\ta(z)+g(x,J y)\ta(J z)\\
\phantom{\W_1:\; F(x,y,z)=\frac{1}{2n}}
    +g(x,z)\ta(y)
    +g(x,J z)\ta(J y)\bigr\};\\
\W_2:\;
F(x,y,J z)+F(y,z,J x)+F(z,x,J y)=0,\quad \ta=0;\\
\W_3:\; F(x,y,z)+F(y,z,z)+F(z,x,y)=0;
\end{array}
\end{equation}
\begin{equation}\label{class2}
\begin{array}{l}
\W_1:\; \Phi(x,y,z)=\frac{1}{2n}\left\{g(x,y)f(z)+g(x,Jy)f(Jz)\right\};\\
\W_2:\;
\Phi(x,y,z)=-\Phi(Jx,Jy,z),\quad f=0;\\
\W_3:\; \Phi(x,y,z)=\Phi(Jx,Jy,z).
\end{array}
\end{equation}
The special class $\W_0$ of the K\"ahler manifolds with Norden
metrics (known also as \emph{K\"ahler-Norden manifolds})
belong to any other class. These manifolds are determined by the condition
$F=0$ and $\Phi=0$, respectively.                                                             

The square norm
$
    \nJ=g^{ij}g^{kl}
        g\bigl(\left(\nabla_{e_i} J\right)e_k,\left(\nabla_{e_j}
    J\right)e_l\bigr)
$
 of $\nabla J$ is defined in \cite{GRMa}.
By means of \eqref{F} and \eqref{F-prop}, we obtain the following equivalent formula
\begin{equation}\label{snorm}
    \nJ=g^{ij}g^{kl}g^{pq}F_{ikp}F_{jlq},
\end{equation}
where $F_{ikp}=F(e_i,e_k,e_p)$.
An almost Norden manifold satisfying the
condition $\nJ=0$ is called an isotropic K\"ahler manifold
with Norden metrics \cite{MekMan} or an \emph{isotropic K\"ahler-Norden manifold}.
Let us remark that if a manifold belongs to $\W_0$, then
it is isotropic K\"ahlerian but the inverse statement is not
always true.


Let $R$ be the curvature tensor field of $\nabla$ defined by
$
    R(x,y)z=\nabla_x \nabla_y z - \nabla_y \nabla_x z -
    \nabla_{[x,y]}z$.
The corresponding tensor field of type $(0,4)$ is determined by
$R(x,y,z,w)=g(R(x,y)z,w)$. It has the following properties:
\begin{equation}\label{curv}
\begin{array}{l}%
R(x,y,z,w)=-R(y,x,z,w)=-R(x,y,w,z),\\
R(x,y,z,w)+R(y,z,x,w)+R(z,x,y,w)=0.
\end{array}%
\end{equation}
Any tensor of type (0,4) satisfying \eqref{curv}
is called a \emph{curvature-like tensor}.
The Ricci tensor $\rho$ and the scalar
curvature $\tau$ for $R$ (and similarly for every curvature-like tensor)
are defined as usual by
$\rho(y,z)=g^{ij}R(e_i,y,z,e_j)$ and $\tau=g^{ij}\rho(e_i,e_j)$.

It is well-known that the Weyl tensor $W$ on a pseudo-Rie\-mann\-ian manifold $(M,g)$, $\dim{M}=2n\geq 4$,
is given by
\begin{equation}\label{W}
W=R+\frac{1}{2(n-1)}g\owedge\rho-\frac{\tau}{4(n-1)(2n-1)}g\owedge g,
\end{equation}
where $g\owedge\rho$ is the Kulkarni-Nomizu product of $g$ and $\rho$, i.e.
$\left(g\owedge\rho\right)(x,y,z,w)=g(x,z)\rho(y,w)-g(y,z)\rho(x,w)
+g(y,w)\rho(x,z)-g(x,w)\rho(y,z).$
Moreover, $W$ vanishes if and only if the manifold $(M,g)$
is conformally flat, i.e. it is transformed into a flat manifold by an usual conformal transformation of the metric defined by $\widebar{g}=e^{2u}g$ for a differentiable function $u$ on $M$.

Let $\tR$ be the curvature tensor of $\tn$ defined as usually.
Obviously, the corresponding curvature (0,4)-tensor is
$\tR(x,y,z,u)=\tg(\tR(x,y)z,u)$ and it has the same properties as in \eqref{curv}.
The Weyl tensor $\widetilde{W}$ is generated by $\tn$ and $\tg$ by the same way and it has the same geometrical interpretation for the manifold $(M,J,\tg)$.


\section{The twin interchange
corresponding to the pair of Norden metrics and their Levi-Civita connections}



\subsection{Invariant classification}

\begin{lem}\label{lem:Phi}
  The potential $\Phi(x,y)$ is an anti-invariant tensor under the twin interchange, i.e.
\begin{equation}\label{tP=-P}
    \tP(x,y)=-\Phi(x,y).
\end{equation}
\end{lem}
\begin{proof}
The equalities \eqref{F-prop}, 
\eqref{tn=nPhi},  \eqref{Phi03} and \eqref{PhiF} imply the following relation between $F$ and its corresponding tensor $\tF$ for $(M,J,\tg)$, defined by $\tF(x,y,z)=\tg\bigl(\bigl(\tn_xJ\bigr)y,z\bigr)$,
\begin{equation}\label{tFF}
    \tF(x,y,z)=\frac{1}{2}\bigl\{F(Jy,z,x)-F(y,Jz,x)+F(Jz,y,x)-F(z,Jy,x)\bigr\}.
\end{equation}
Bearing in mind \eqref{PhiF}, we write the corresponding formula for $\tP$ and $\tF$ as
\begin{equation}\label{tPtF}
    \tP(x,y,z)=\frac{1}{2}\bigl\{\tF(Jz,x,y)-\tF(x,y,Jz)-\tF(y,x,Jz)\bigr\}.
\end{equation}
Using \eqref{tFF} and \eqref{tPtF}, we get an expression of $\tP$ in terms of $F$ and then by \eqref{PhiF} we obtain
\begin{equation}\label{tPPhi}
\tP(x,y,z)=-\Phi(x,y,Jz).
\end{equation}
Taking into account that $\tP(x,y,z)$ is defined by $\tP(x,y,z)=\tg(\tP(x,y),z)$, we accomplish the proof.
\end{proof}

In \cite{GaGrMi}, for an arbitrary almost Norden manifold, it is given the following identity
\begin{equation}\label{Phi-prop}
  \Phi(x,y,z)-\Phi(Jx,Jy,z)-\Phi(Jx,y,Jz)-\Phi(x,Jy,Jz)=0.
\end{equation}
The associated 1-forms $f$ and $f^*$ of $\Phi$ are defined by $f(z)=g^{ij}\Phi(e_i,e_j,z)$ and $f^*(z)=g^{ij}\Phi(e_i,Je_j,z)$. Obviously, $f(z)=g(\mathrm{tr}{\Phi},z)$ holds.
Then, from \eqref{Phi-prop} we get the identity
\begin{equation}\label{f-prop}
f(z)=f^*(Jz).
\end{equation}
The latter identity resembles the equality $\ta(z)=\ta^*(Jz)$, equivalent to \eqref{ta*taJ}.
Indeed, there exists a relation between the associated 1-forms of $F$ and $\Phi$. It follows from \eqref{PhiF} and has the form
\begin{equation}\label{fta}
f(z)=\ta^*(z),\quad f^*(z)=-\ta(z).
\end{equation}

\begin{lem}\label{lem:f}
The associated 1-forms $f$ and $f^*$ of $\Phi$ are invariant under the twin interchange, i.e.
\begin{equation*}\label{tff}
\widetilde{f}(z)=f(z), \qquad \widetilde{f}^*(z)=f^*(z).
\end{equation*}
\end{lem}
\begin{proof}
Taking the trace of \eqref{tPPhi} by $\tg^{ij}=-J^j_kg^{ik}$ for $x=e_i$ and $y=e_j$, we have $\widetilde{f}(z)=f^*(Jz)$. Then, comparing the latter equality and \eqref{f-prop}, we obtain the statement for $f$. The relation in the case of $f^*$ is valid because of \eqref{f-prop}.
\end{proof}

\begin{lem}\label{lem:tata*}
The Lee forms $\ta$ and $\ta^*$ are invariant under the twin interchange, i.e.
\[
\ta(z)=\widetilde{\ta}(z),\qquad \ta^*(z)=\widetilde{\ta}^*(z).
\]
\end{lem}
\begin{proof} It follows directly from \lemref{lem:f} and \eqref{fta}.
\end{proof}

\begin{thm}\label{thm:inv.cl}
  All classes of almost Norden manifolds according to the classification in \cite{GaBo} are invariant under the twin interchange.
\end{thm}
\begin{proof}
We use the classification by $\Phi$ in \cite{GaGrMi}: the definitions of the basic three classes are given in \eqref{class2}, for the special class and the whole class  we have $\W_0:\; \Phi=0$ $\W_1\oplus\W_2\oplus\W_3:\;$\emph{no condition} as well as the rest classes are defined as follows:
\[
\begin{array}{l}
\W_1\oplus\W_2:\; \Phi(x,y,z)=-\Phi(Jx,Jy,z),\qquad \W_2\oplus\W_3:\; f=0,\\ \W_1\oplus\W_3:\;\Phi(x,y,z)-\Phi(Jx,Jy,z)=\frac{1}{n}\left\{g(x,y)f(z)+g(x,Jy)f(Jz)\right\}.
\end{array}
\]

Obviously, applying \lemref{lem:Phi}, \lemref{lem:f}, equalities \eqref{tPPhi} and \eqref{Phi03}, we establish the truthfulness of the statement.
\end{proof}


Let us remark that the invariance of $\W_1$ and $\W_3$ is proved in \cite{GaGrMi} and \cite{MekManGri22}, respectively.

\subsection{Invariant connection}

Let us define a linear connection $D$ by
\begin{equation}\label{hn=nP}
D_x y=\n_x y+\frac12\Phi(x,y).
\end{equation}
By virtue of \eqref{tn=nPhi}, \eqref{tP=-P} and \eqref{hn=nP}, we have the following
\[
\widetilde{D}_x y=\tn_x y+\frac12\tP(x,y)=\n_x y+\Phi(x,y)-\frac12\Phi(x,y)
=\n_x y+\frac12\Phi(x,y)=D_x y.
\]
Therefore, $D$ is an invariant connection under the twin interchange.
Bearing in mind \eqref{Phi}, we establish that $D$ is actually the \emph{average connection} of $\n$ and $\tn$, because
\begin{equation}\label{av.con}
D_x y=\n_x y+\frac12\Phi(x,y)=\n_x y+\frac12\left\{\tn_x y-\n_x y\right\}
=\frac12\left\{\n_x y+\tn_x y\right\}.
\end{equation}
So, we obtain
\begin{prop}\label{prop:inv.conn}
    The average connection $D$ of $\n$ and $\tn$ is an invariant connection under the twin interchange.
\end{prop}

\begin{cor}\label{cor:inv.conn}
    If the invariant connection $D$ vanishes then $(M,J,g)$ and $(M,J,\tg)$ are K\"ahler-Norden manifolds and $\n=\tn$ also vanishes.
\end{cor}
\begin{proof}
Let us suppose that $D=0$. Then $\n=-\tn$ and $\Phi=-2\n$, because of \eqref{hn=nP} and \eqref{av.con}. Hence we obtain $[x,y]=\n_x y-\n_y x=-\frac12\{\Phi(x,y)-\Phi(y,x)\}=0$ and consequently, using the Koszul formula
\[
2g(\nabla_x y,z)=xg(y,z)+yg(x,z)-zg(x,y)+g([x,y],z)+g([z,x],y)+g([z,y],x),
\]
we get $\n=0$. Thus, $\tn$ and $\Phi$ vanish, i.e. $(M,J,g)$ and $(M,J,\tg)$ belong to $\W_0$.
\end{proof}

\subsection{Invariant tensors}

As it is well-known, the Nijenhuis tensor $N$ of the almost complex structure $J$ is
defined by
\begin{equation*}\label{NJ}
N(x,y) = [J, J](x, y)=\left[Jx,Jy\right]-\left[x,y\right]-J\left[Jx,y\right]-J\left[x,Jy\right].
\end{equation*}
Besides $N$, in \cite{Man50} it is defined the following symmetric
(1,2)-tensor $S$  in analogy by
\[
S(x,y)=\{J ,J\}(x,y)=\{Jx,Jy\}-\{x,y\}-J\{Jx,y\}-J\{x,Jy\},
\]
where the symmetric braces $\{x,y\}=\nabla_xy+\nabla_yx$
are used instead of the antisymmetric brackets $[x,y]=\nabla_xy-\nabla_yx$.
The tensor $S$ is also called the  \emph{associated
Nijenhuis tensor} of $J$. The tensor $S$ coincides with the associated tensor of ${N}$ introduced in \cite{GaBo} by an equivalent equality for $F$.

\begin{prop}\label{prop:inv.NhN}
    The Nijenhuis tensor is invariant and the associated Nijenhuis tensor is anti-invariant under the twin interchange, i.e.
    \[
    N(x,y)=\tN(x,y),\qquad S(x,y)=-\widetilde{S}(x,y).
    \]
\end{prop}
\begin{proof}
The relations of $N$ and $S$ with $\Phi$ are given in \cite{GaGrMi} as follows
\begin{gather}
  N(x,y,z)=2\Phi(z,Jx,Jy)-2\Phi(z,x,y),\label{NPhi}\\
  S(x,y,z)=2\Phi(x,y,z)-2\Phi(Jx,Jy,z).\label{wNPhi}
\end{gather}
Using \eqref{tPPhi}, the latter equalities imply the following
\begin{gather}
  \tN(x,y,z)=-N(x,Jy,z),\label{tNN}\\
  \tS(x,y,z)=-S(x,y,Jz).\label{twNwN}
\end{gather}
In \cite{Man50}, it is given the property $N(x, y,z) = N(x, Jy, Jz)$ which is equivalent to $N(x,Jy,z) = -N(x,y, Jz)$. Then \eqref{tNN} gets the form
\begin{gather}
  \tN(x,y,z)=N(x,y,Jz).\label{tNN2}
\end{gather}
The equalities \eqref{tNN2} and \eqref{twNwN} yield the relations in the statement.
\end{proof}

It is
well-known the following relation between the curvature tensors of $\n$ and $\tn$ related by \eqref{tn=nPhi}
\begin{equation}\label{tRRS}
    \tR(x,y)z=R(x,y)z+Q(x,y)z,
\end{equation}
where
\begin{equation}\label{Q}
    Q(x,y)z= \left(\n_x \Phi\right)(y,z)- \left(\n_y \Phi\right)(x,z)
    +\Phi\left(x,\Phi(y,z)\right)-\Phi\left(y,\Phi(x,z)\right).
\end{equation}

Let us consider the following tensor, which is part of the tensor $Q$,
\begin{equation}\label{A}
A(x,y)z=\Phi(x,\Phi(y,z))-\Phi(y,\Phi(x,z)).
\end{equation}
\begin{lem}\label{lem:A}
  The tensor $A(x,y)z$ is invariant under the twin interchange, i.e.
\begin{equation}\label{A=tA}
A(x,y)z=\tA(x,y)z.
\end{equation}
\end{lem}
\begin{proof}
Since \eqref{tP=-P} is valid, we obtain immediately
\begin{equation*}\label{nPtnP}
\Phi(x,\Phi(y,z))-\Phi(y,\Phi(x,z))=\tP(x,\tP(y,z))-\tP(y,\tP(x,z)),
\end{equation*}
which yields relation \eqref{A=tA}.
\end{proof}

\begin{lem}\label{lem:Q}
  The tensor $Q(x,y)z$ is anti-invariant under the twin interchange, i.e.
\begin{equation}\label{tS=-Q}
    \tQ(x,y)z=-Q(x,y)z.
\end{equation}
\end{lem}
\begin{proof}
For the covariant derivative of $\Phi$ we have
\[
\left(\n_x\Phi\right)(y,z)=\n_x\Phi(y,z)-\Phi(\n_x y,z)-\Phi(y,\n_x z).
\]
Applying \eqref{tn=nPhi} and \eqref{tP=-P}, we get
\[
\left(\n_x\Phi\right)(y,z)=-(\tn_x\tP)(y,z)-\tP(x,\tP(y,z))+\tP(y,\tP(x,z))+\tP(z,\tP(x,y)).
\]
As a sequence of the latter equality and \eqref{A} we obtain
\begin{equation}\label{nP}
\begin{array}{l}
    \left(\n_x\Phi\right)(y,z)-\left(\n_y\Phi\right)(x,z)=
    -(\tn_x\tP)(y,z)+(\tn_y\tP)(x,z)-2\widetilde{A}(x,y)z.
\end{array}
\end{equation}
Then, \eqref{Q}, \eqref{A}, \eqref{A=tA} and \eqref{nP} imply relation \eqref{tS=-Q}.
\end{proof}

\begin{prop}\label{prop:inv.tensor2}
    The curvature tensor $K$ of the average connection $D$ for $\n$ and $\tn$ is an invariant tensor under the twin interchange.
\end{prop}
\begin{proof}
From \eqref{hn=nP}, using the formulae
\eqref{tn=nPhi}, \eqref{tRRS}, \eqref{Q} and \eqref{A}, we get the following relation
\begin{equation*}\label{hRRS}
\begin{array}{l}
    K(x,y)z=R(x,y)z+\frac12\left(\n_x\Phi\right)(y,z)-\frac12\left(\n_y\Phi\right)(x,z)
    +\frac14 A(x,y)z,
\end{array}
\end{equation*}
which is actually
\begin{equation}\label{hRRSA}
    K(x,y)z=R(x,y)z+\frac12Q(x,y)z
    -\frac14 A(x,y)z.
\end{equation}
By virtue of \eqref{tRRS}, \eqref{A=tA}, \eqref{tS=-Q} and \eqref{nP}, we establish the relation $\widetilde{K}=K$.
\end{proof}

As a consequence of \eqref{hRRSA}, we obtain the following
\begin{cor}\label{cor:inv.tensor2}
    The invariant tensor $K$ vanishes if and only if
    \[
    R(x,y)z=-\frac12 Q(x,y)z+\frac14 A(x,y)z.
    \]
\end{cor}

Let us consider the average tensor $P$ of the curvature tensors $R$ and $\tR$, respectively, i.e.
$P(x,y)z=\frac12\{R(x,y)z+\tR(x,y)z\}$.
Then by \eqref{tRRS} we have
\begin{equation}\label{bR=RS}
    P(x,y)z=R(x,y)z+\frac12 Q(x,y)z.
\end{equation}

\begin{prop}\label{prop:inv.tensor}
    The average tensor $P$ of $R$ and $\tR$ is an invariant tensor under the twin interchange, i.e. $P(x,y)z=\widetilde{P}(x,y)z$.
\end{prop}
\begin{proof}
Using \eqref{tRRS}, \eqref{bR=RS} and \eqref{tP=-P},  we have the following
\begin{equation*}\label{tbR=bR}
\begin{array}{l}
    \widetilde{P}(x,y)z=\tR(x,y)z+\frac12 \tQ(x,y)z\\
    \phantom{\widetilde{P}(x,y)z}
    =R(x,y)z+Q(x,y)z-\frac12 Q(x,y)z\\
    \phantom{\widetilde{P}(x,y)z}
    =R(x,y)z+\frac12 Q(x,y)z=P(x,y)z.
\end{array}
\end{equation*}
\vskip-1em
\end{proof}

Immediately from \eqref{bR=RS} we obtain the following
\begin{cor}\label{cor:bR=0}
    The invariant tensor $P$ vanishes if and only if $R=-\frac12Q$ is valid.
\end{cor}

By virtue of \eqref{hRRSA} and \eqref{bR=RS}, we have the following relation between the invariant tensors $K$, $P$ and $A$
\begin{equation}\label{wRbRPhi}
    K(x,y)z=P(x,y)z
    -\frac14 A(x,y)z.
\end{equation}

\begin{thm}\label{thm:inv.tensors}
    Any linear combination of the invariant tensors $P$ and $K$ is an invariant tensor under the twin interchange.
\end{thm}
\begin{proof}
It follows from \propref{prop:inv.tensor2} and \propref{prop:inv.tensor}.
\end{proof}

\subsection{Invariant connection and invariant tensors on the manifolds in the main class}

Now, we consider an arbitrary manifold $(M,J,g)$ belonging to the basic class $\W_1$. This class is known as the main class in the classification in \cite{GaBo}, because it is the only class where the fundamental tensor $F$ and the potential $\Phi$ are expressed explicitly by the metric. Then, we have the form of $F$ and $\Phi$ in \eqref{class} and \eqref{class2}, respectively.
Taking into account \eqref{tFF}, \eqref{class} and \eqref{F-prop},
we obtain the following form of  $F$ 
under the twin interchange
\begin{equation*}\label{W1:tF}
\begin{array}{l}
  \tF(x,y,z)=-\frac{1}{2n}\bigl\{
  g(x,y)\ta(Jz)+g(x,z)\ta(Jy)\\
  \phantom{\tF(x,y,z)=-\frac{1}{2n}\bigl\{}
  -g(x,Jy)\ta(z)-g(x,Jz)\ta(y)\bigr\}.
\end{array}
\end{equation*}
Therefore, we get the following relation for a $\W_1$-manifold
\begin{equation*}\label{W1:tFF}
  \tF(x,y,z)=F(Jx,y,z).
\end{equation*}

%
%

The invariant connection on a $\W_1$-manifold has the following form, applying the definition from \eqref{class2} in \eqref{hn=nP},
\begin{equation*}\label{W1:inv-n}
D_x y=\n_x y+\frac{1}{4n}\left\{g(x,y)f^{\sharp}+g(x,Jy)Jf^{\sharp}\right\},
\end{equation*}
where $f^{\sharp}$ is the dual vector of the 1-form $f$ regarding $g$, i.e. $f(z)=g(f^{\sharp},z)$.

The presence of the first equality in \eqref{class2}, the explicit expression of $\Phi$ in terms of $g$ for the case of a $\W_1$-manifold,
gives us a chance to find a more concrete form of $Q$ and $A$ defined by \eqref{Q} and  \eqref{A}, respectively. This expression gives results in the corresponding relations between $R$ and $\tR$, $K$, $P$, given in \eqref{tRRS}, \eqref{hRRSA}, \eqref{bR=RS}, respectively. A relation $R$ and $\tR$ for a $\W_1$-manifold is given in \cite{MT06th} but using $\ta$.

\begin{prop}\label{prop:W1_R}
If $(M,J,g)$ is an almost Norden manifold belonging to the class $\W_1$, then the tensors $Q$ and $A$ have the following form, respectively:
\begin{equation*}\label{W1:Q}
\begin{array}{l}
Q(x,y)z=\frac{1}{2n}\bigl\{g(y,z)p(x)+g(y,Jz)Jp(x)\\
\phantom{Q(x,y)z=\frac{1}{2n}}
-g(x,z)p(y)-g(x,Jz)Jp(y)\bigr\},
\end{array}
\end{equation*}
\begin{equation*}\label{W1:A}
\begin{array}{l}
A(x,y)z=\frac{1}{4n^2}\bigl\{g(y,z)h(x)+g(y,Jz)h(Jx)\\
\phantom{A(x,y)z=\frac{1}{4n^2}}
-g(x,z)h(y)-g(x,Jz)h(Jy)\bigr\},
\end{array}
\end{equation*}
where $p(x)=\n_x f^{\sharp}+\frac{1}{2n}\{f(x)f^{\sharp}-f(f^{\sharp})x-f(Jf^{\sharp})Jx\}$ and $h(x)=f(x)f^{\sharp}+f(Jx)Jf^{\sharp}$.
\end{prop}
\begin{proof}
The formulae follow by direct computations, using \eqref{class}, \eqref{class2}, \eqref{Q} and \eqref{A}.
\end{proof}


\section{Lie group as a manifold from the main class and the invariant connection and the invariant tensors on it}
\label{sec_3}

In this section we consider an example of a 4-dimensional Lie
group as a $\W_1$-manifold given in \cite{MT06ex}.

Let $L$ be a 4-dimensional real connected Lie group, and
$\mathfrak{l}$ be its Lie algebra with a basis
$\{X_{1},X_{2},X_{3},X_{4}\}$.

We introduce an almost complex structure
$J$ and a Norden metric by
\begin{equation}\label{Jdim4}
\begin{array}{llll}
JX_{1}=X_{3}, \quad & JX_{2}=X_{4}, \quad & JX_{3}=-X_{1},
\quad &
JX_{4}=-X_{2},
\end{array}
\end{equation}
\begin{equation}\label{g}
\begin{array}{c}
  g(X_1,X_1)=g(X_2,X_2)=-g(X_3,X_3)=-g(X_4,X_4)=1, \\
  g(X_i,X_j)=0,\; i\neq j.
\end{array}
\end{equation}
Then, the associated Norden metric $\tg$ is determined by its non-zero components
\begin{equation}\label{tg}
\begin{array}{c}
  \tg(X_1,X_3)=\tg(X_2,X_4)=-1.
\end{array}
\end{equation}

Let us consider $(L,J,g)$ with the Lie algebra $\mathfrak{l}$
determined by the following nonzero commutators:
\begin{equation}\label{lie-w1-2}
\begin{array}{l}
\left
[X_{1},X_{4}\right]=[X_{2},X_{3}]=\lm_{1}X_{1}+\lm_{2}X_{2}+\lm_{3}X_{3}+\lm_{4}X_{4},\\
\left[X_{1},X_{3}\right]=[X_{4},X_{2}]=\lm_{2}X_{1}-\lm_{1}X_{2}+\lm_{4}X_{3}-\lm_{3}X_{4},
\end{array}
\end{equation}
where $\lm_i\in\R$ ($i=1,2,3,4$).
Obviously, $[J X_i,J
X_j]=[X_i,X_j]$ holds, i.e. $J$ is an Abelian structure for
$\mathfrak{l}$.

In \cite{MT06ex}, it is proved that $(L,J,g)$ is a $\W_1$-manifold.
Since the class $\W_1$ is invariant under the twin interchange, it follows that $(L,J,\tg)$ belongs to  $\W_1$, too.

\begin{thm}\label{thm:W10-W}
Let $(L,J,g)$ and $(L,J,\tg)$ be the pair of $\W_1$-manifolds, determined by
\eqref{Jdim4}--\eqref{lie-w1-2}. Then both the manifolds:
\begin{enumerate}\renewcommand{\labelenumi}{(\roman{enumi})}
    \item belong to the class of the locally conformal
        K\"ahler-Norden manifolds if and only if
        \begin{equation}\label{lm}
        \lm_1^2-\lm_2^2+\lm_3^2-\lm_4^2=\lm_1\lm_2+\lm_3\lm_4=0;
        \end{equation}
    \item   are 
            locally conformally flat by usual conformal transformations
            and the curvature tensors $R$ and $\tR$ have the following form, respectively:
        \begin{equation}\label{Rform3}
            \begin{array}{l}
                R=-\frac{1}{2}g\owedge\rho+\frac{\tau}{12}g\owedge g,\quad
                              \widetilde{R}=  -\frac{1}{2}\tg\owedge\widetilde{\rho}
                                +\frac{\widetilde{\tau}}{12}\tg\owedge\tg.
            \end{array}
        \end{equation}
    \item   are scalar flat and isotropic K\"ahlerian
            if and only if the following conditions are satisfied, respectively:
        \begin{equation*}\label{llll2}
        \lm_{1}^{2}+\lm_{2}^{2}-\lm_{3}^{2}-\lm_{4}^{2}=0,\quad \lm_{1}\lm_{3}+\lm_{2}\lm_{4}=0.
        \end{equation*}
\end{enumerate}
\end{thm}
\begin{proof}
According to \eqref{twin}, \eqref{Jdim4}, \eqref{g}, \eqref{lie-w1-2} and the Koszul equality for $g$, $\n$ and $\tg$, $\tn$, we obtain the following nonzero components of $\n$ and $\tn$:
\begin{equation}\label{nabla3}
\begin{array}{ll}
\n_{X_{1}}X_{1} = \n_{X_{2}}X_{2} = \tn_{X_{3}}X_{3} = \tn_{X_{4}}X_{4} =\lm_{2}X_{3} +\lm_{1}X_{4},
\\
\n_{X_{1}}X_{3} = \n_{X_{4}}X_{2} =-\tn_{X_{2}}X_{4} = -\tn_{X_{3}}X_{1} = \lm_{2}X_{1} -\lm_{3}X_{4},
\\
\n_{X_{1}}X_{4} = -\n_{X_{3}}X_{2} =\tn_{X_{1}}X_{4} = -\tn_{X_{3}}X_{2} = \lm_{1}X_{1} +\lm_{3}X_{3},
\\
\n_{X_{2}}X_{3} =- \n_{X_{4}}X_{1} =\tn_{X_{2}}X_{3} =- \tn_{X_{4}}X_{1} =\lm_{2}X_{2} +\lm_{4}X_{4},
\\
\n_{X_{2}}X_{4} = \n_{X_{3}}X_{1} =-\tn_{X_{1}}X_{3} = -\tn_{X_{4}}X_{2} = \lm_{1}X_{2} -\lm_{4}X_{3},
\\
\n_{X_{3}}X_{3} = \n_{X_{4}}X_{4} =\tn_{X_{1}}X_{1} = \tn_{X_{2}}X_{2} = -\lm_{4}X_{1} -\lm_{3}X_{2}.
\end{array}
\end{equation}

The components of $\n J$ and $\tn J$ follow from \eqref{nabla3} and \eqref{Jdim4}. Then, using \eqref{g}, \eqref{tg} and \eqref{F}, we
get the following nonzero components
$F_{ijk}=F(X_{i},X_{j},X_{k})$ and $\tF_{ijk}=\tF(X_{i},X_{j},X_{k})=\tg((\tn_{X_i}J)X_j,X_k)$
of $F$ and $\tF$, respectively:
\begin{equation}\label{lambdi}
\begin{array}{l}
\lm_{1}=F_{112}=F_{121}=F_{134}=F_{143}=\frac{1}{2}F_{222}\\
\phantom{\lm_{1}}=\frac{1}{2}F_{244}=F_{314}
=-F_{323}=-F_{332}=F_{341},\\
\lm_{2}=\frac{1}{2}F_{111}=\frac{1}{2}F_{133}=F_{212}=F_{221}=F_{234}\\
\phantom{\lm_2}=F_{243}=-F_{414}=F_{423}=F_{432}
=-F_{441},\\
\lm_{3}=F_{114}=-F_{123}=-F_{132}=F_{141}=-F_{312}\\
\phantom{\lm_{3}}=-F_{321}=-F_{334}=-F_{343}=-\frac{1}{2}F_{422}
=-\frac{1}{2}F_{444},\\
\lm_{4}=-F_{214}=F_{223}=F_{232}=-F_{241}=-\frac{1}{2}F_{311}\\
\phantom{\lm_{4}}=-\frac{1}{2}F_{333}=-F_{412}=-F_{421}=-F_{434}
=-F_{443};
\end{array}
\end{equation}
\begin{equation}\label{lambdi-tF}
\begin{array}{l}
\lm_{1}=\tF_{114}=-\tF_{123}=-\tF_{132}=\tF_{141}=-\tF_{312}=-\tF_{321}\\
\phantom{\lm_{1}}
=-\tF_{334}=-\tF_{343}
=-\frac12\tF_{422}=-\frac12\tF_{444},\\
\lm_{2}=-\tF_{214}=\tF_{223}=\tF_{232}=-\tF_{241}=-\frac12\tF_{311}\\
\phantom{\lm_2}
=-\frac12\tF_{333}=-\tF_{412}=-\tF_{421}=-\tF_{434}
=-\tF_{443},\\
\lm_{3}=-\tF_{112}=-\tF_{121}
=-\tF_{134}=-\tF_{143}=-\frac12\tF_{222}\\
\phantom{\lm_{3}}
=-\frac12\tF_{244}=-\tF_{314}=\tF_{323}=\tF_{332}
=-\tF_{341},\\
\lm_{4}=-\frac12\tF_{111}=-\frac12\tF_{133}=-\tF_{212}=-\tF_{221}=-\tF_{234}\\
\phantom{\lm_{4}}
=-\tF_{243}=\tF_{414}=-\tF_{423}=-\tF_{432}
=\tF_{441}.
\end{array}
\end{equation}

Applying \eqref{snorm} for the components in \eqref{lambdi} and \eqref{lambdi-tF}, we obtain the square norms of $\n J$ and $\tn J$:
\begin{equation}\label{nJ}
\nJ=16\left(\lm_1^2+\lm_2^2-\lm_3^2-\lm_4^2\right),\quad \tnJ=\left(\lm_1\lm_3+\lm_2\lm_4\right).
\end{equation}

Let us consider the conformal transformations $\widebar{g}=e^{2u}(\cos{2v}\ g+\sin{2v}\ \tg)$ of the metric $g$, where $u$ and $v$ are differentiable functions on the manifold.
Then, the associated metric $\tg$ has the following image $\widebar{\tg}=e^{2u}(\cos{2v}\ \tg-\sin{2v}\ g)$.
If $v=0$, we obtain the usual conformal transformation.
Let us remark that the conformal transformation for $u=0$ and $v=\pi/2$ maps the pair $(g,\tg)$ into $(\tg,-g)$.

According to \cite{GaGrMi}, a $\W_1$-manifold is locally conformal equivalent to a K\"ahler-Norden manifold if and only if its Lee forms $\ta$ and $\ta^*$ are closed. Moreover, the used conformal transformations are such that the 1-forms $\D u\circ J$ and $\D v\circ J$ are closed.

Taking into account \lemref{lem:tata*}, we have
$\ta_k=\widetilde{\ta}_k$ and
$\ta^*_k=\widetilde{\ta}^*_k$ for the corresponding components with respect to $X_k$.
Furthermore, the same situation is for  $\D{\ta}=\D\widetilde{\ta}$ and $\D{\ta}^*=\D\widetilde{\ta}^*$.
By \eqref{ta}, \eqref{ta*taJ} and \eqref{lambdi}, we obtain
$\ta_k$ and $\ta^*_k$ and thus we get:
\begin{equation}\label{theta123}
\begin{array}{ll}
\ta_{1}=\ta^*_{3}=\widetilde{\ta}_{1}=\widetilde{\ta}^*_{3}=4\lm_{2}, \quad &
\ta_{3}=-\ta^*_{1}=\widetilde{\ta}_{3}=-\widetilde{\ta}^*_{1}=4\lm_{4}, \quad \\
\ta_{2}=\ta^*_{4}=\widetilde{\ta}_{2}=\widetilde{\ta}^*_{4}=4\lm_{1}, \quad &
\ta_{4}=-\ta^*_{2}=\widetilde{\ta}_{4}=-\widetilde{\ta}^*_{2}=4\lm_{3}.
\end{array}
\end{equation}
Using \eqref{lie-w1-2} and \eqref{theta123}, we compute the components of
$\D\ta$ and $\D\ta^*$ with respect to the basis
$\{X_{1},X_{2},X_{3},X_{4}\}$.
We obtain that $\D\ta^*=\D\widetilde{\ta}^*=0$ and the nonzero components of $\D\ta=\D\widetilde{\ta}$ are
\begin{equation*}\label{dta}
\begin{array}{l}
  \D\ta_{13}=\D\ta_{42}=\D\widetilde{\ta}_{13}=\D\widetilde{\ta}_{42}=4(\lm_1^2-\lm_2^2+\lm_3^2-\lm_4^2),\\
  \D\ta_{14}=\D\ta_{23}=\D\widetilde{\ta}_{14}=\D\widetilde{\ta}_{23}=-8(\lm_1\lm_2+\lm_3\lm_4).
\end{array}
\end{equation*}

Therefore $(L,J,g)$ and $(L,J,\tg)$ and
locally conformal K\"ahler-Norden manifolds if and only if conditions
\eqref{lm} are valid. Then, the statement (i) holds.

By virtue of \eqref{g}, \eqref{lie-w1-2} and \eqref{nabla3}, we get  $R_{ijkl}=R(X_{i},X_{j},X_{k},X_{l})$ and $\tR_{ijkl}=\tR(X_{i},\allowbreak{}X_{j},\allowbreak{}X_{k},X_{l})$, the basic components of the curvature tensors
for $\n$ and $\tn$. The nonzero ones of them are determined by \eqref{curv} and the following:
\begin{equation}\label{R3}
\begin{array}{l} 
\begin{array}{lll}
R_{1221} = \lm_{1}^{2} + \lm_{2}^{2}, \qquad & %
R_{1331} = \lm_{4}^{2} - \lm_{2}^{2}, \qquad & %
R_{1441} = \lm_{4}^{2} - \lm_{1}^{2},
\\
R_{2332} = \lm_{3}^{2} - \lm_{2}^{2},\qquad & %
R_{2442} = \lm_{3}^{2} - \lm_{1}^{2}, \qquad & %
R_{3443} = -\lm_{3}^{2} - \lm_{4}^{2},\\
\end{array}
\\
\begin{array}{ll}
R_{1341}=R_{2342} = -\lm_{1}\lm_{2}, \qquad & %
R_{2132}=-R_{4134} = -\lm_{1}\lm_{3},
\\
R_{1231}=-R_{4234} = \lm_{1}\lm_{4}, \qquad & %
R_{2142}=-R_{3143} = \lm_{2}\lm_{3},
\\
R_{1241}=-R_{3243} = -\lm_{2}\lm_{4}, \qquad & %
R_{3123}=R_{4124} = \lm_{3}\lm_{4};
\end{array}
\end{array}
\end{equation}
\begin{equation}\label{tR3}
\begin{array}{l}
\begin{array}{lll}
\tR_{1241} = -\lm_{3}^{2}, \quad & %
\tR_{2132} = -\lm_{4}^{2},
\\
\tR_{1331} = 2\lm_{2}\lm_{4},\quad & %
\tR_{2442} = 2\lm_{1}\lm_{3},
\end{array}
\\
\begin{array}{ll}
\tR_{3143}=\tR_{4234} = -\lm_{1}\lm_{2},
\\
\tR_{1231}=\tR_{2142} = -\lm_{3}\lm_{4},
\end{array}
\end{array}
\begin{array}{l}
\tR_{3243} = -\lm_{1}^{2}, \qquad  %
\tR_{4134} = -\lm_{2}^{2},
\\
\tR_{1234} = \tR_{2341} = \lm_{1}\lm_{3} + \lm_{2}\lm_{4},
\\
\tR_{1341}=\tR_{4124} = \lm_{2}\lm_{3},
\\
\tR_{2342}=\tR_{3123} = \lm_{1}\lm_{4}.
\end{array}
\end{equation}
Therefore, the components of the Ricci tensors
and the values of the scalar curvatures
for $\n$ and $\tn$ are:
\begin{equation}\label{Ricci3}
\begin{array}{c}
\begin{array}{ll}
\rho_{11}=2\big( \lm_{1}^{2} + \lm_{2}^{2} - \lm_{4}^{2} \big), \qquad &
\rho_{22}=2\big( \lm_{1}^{2} + \lm_{2}^{2} - \lm_{3}^{2} \big), \\
\rho_{33}=2\big( \lm_{4}^{2} + \lm_{3}^{2} - \lm_{2}^{2} \big), \qquad &
\rho_{44}=2\big( \lm_{4}^{2} + \lm_{3}^{2} - \lm_{1}^{2} \big),
\\
\widetilde{\rho}_{11}=2\lm_{3}^{2}, \qquad\;
\widetilde{\rho}_{22}=2\lm_{4}^{2}, \qquad &
\widetilde{\rho}_{33}=2\lm_{1}^{2},  \qquad \;
\widetilde{\rho}_{44}=2\lm_{2}^{2},
\\
\widetilde{\rho}_{13}=2\big(\lm_{1}\lm_{3}+2\lm_{2}\lm_{4}\big),\qquad &
\widetilde{\rho}_{24}=2\big(2\lm_{1}\lm_{3}+\lm_{2}\lm_{4}\big),
\end{array}
\\
\begin{array}{lll}
\rho_{12}=\widetilde{\rho}_{12}=-2\lm_{3}\lm_{4}, \quad & %
\rho_{23}=\widetilde{\rho}_{23}=2\lm_{1}\lm_{4}, \quad &
\rho_{13}=-2\lm_{1}\lm_{3},
\\
\rho_{34}=\widetilde{\rho}_{34}=-2\lm_{1}\lm_{2}, \quad & %
\rho_{14}=\widetilde{\rho}_{14}=2\lm_{2}\lm_{3}, \quad &
\rho_{24}=-2\lm_{2}\lm_{4},
\end{array}
\\
\begin{array}{ll}
\tau=6\big(\lm_{1}^{2}+\lm_{2}^{2}-\lm_{3}^{2}-\lm_{4}^{2}\big),\qquad &
\widetilde{\tau}=-12\big(\lm_{1}\lm_{3}+\lm_{2}\lm_{4}\big).
\end{array}
\end{array}
\end{equation}

Applying \eqref{W} for the corresponding quantities of $\n$ and $\tn$, we compute that the Weyl tensors $W$ and $\widetilde{W}$ vanish, respectively. Then, we obtain the identities in \eqref{Rform3}.
Furthermore, when the Weyl tensor vanishes then the corresponding manifold is conformal equivalent to a flat manifold by a usual conformal transformation. This completes the proof of (ii).

The truthfulness of (iii) follows immediately from the equations in the last line of \eqref{Ricci3} and the values of the square norms in \eqref{nJ}. %
\end{proof}

Let us remark that the results in the latter theorem with respect to $\n$ 
are given in \cite{MT06ex} besides (i), where it is shown a particular case of conditions \eqref{lm}.

\subsection{The invariant connection and invariant tensors under the twin interchange }

We compute the basic components $P_{ijk}=P(X_i,X_j)X_k$ of the invariant tensor $P$, using \eqref{R3}, \eqref{tR3} and that this tensor is the average tensor of $R$ and $\tR$, and get the components $P_{ijkl}=g\left(P(X_i,X_j)X_k,X_l\right)$:
\begin{subequations}\label{bR3}
\begin{equation}
\begin{array}{l}
\begin{array}{ll}
\frac12\lm_{1}^{2} = P_{3421} = -P_{2341},\quad &
\frac12\lm_{2}^{2} = -P_{3412} = -P_{1432},\\
\frac12\lm_{3}^{2} = -P_{1243} = -P_{1423},\quad &
\frac12\lm_{4}^{2} = P_{1234} = -P_{2314}, \\
\frac12\bigl(\lm_{2}^{2}-\lm_{4}^{2}\bigr) = P_{1313} = -P_{1331}, \quad &
\frac12\bigl(\lm_{1}^{2}-\lm_{3}^{2}\bigr) = P_{2424} = -P_{2442}, \\
\frac12\bigl(\lm_{1}^{2}+\lm_{2}^{2}+\lm_{3}^{2}\bigr) = -P_{1212}, \quad &
\frac12\bigl(\lm_{1}^{2}+\lm_{2}^{2}+\lm_{4}^{2}\bigr) = P_{1221},\\
\frac12\bigl(\lm_{1}^{2}+\lm_{3}^{2}-\lm_{4}^{2}\bigr) = P_{1414}, \quad &
\frac12\bigl(\lm_{1}^{2}-\lm_{2}^{2}-\lm_{4}^{2}\bigr) = -P_{1441},\\
\frac12\bigl(\lm_{2}^{2}-\lm_{3}^{2}+\lm_{4}^{2}\bigr) = P_{2323}, \quad &
\frac12\bigl(\lm_{1}^{2}-\lm_{2}^{2}+\lm_{3}^{2}\bigr) = P_{2332},\\
\frac12\bigl(\lm_{1}^{2}+\lm_{3}^{2}+\lm_{4}^{2}\bigr) = P_{3434}, \quad &
\frac12\bigl(\lm_{2}^{2}+\lm_{3}^{2}+\lm_{4}^{2}\bigr) = -P_{3443},
\end{array}
\\
\begin{array}{l}
\frac12\bigl(\lm_{1}\lm_{2}+\lm_{3}\lm_{4}\bigr) = P_{1234} = P_{1332} = P_{2423} = P_{2441}, \\
\frac12\bigl(\lm_{2}\lm_{3}-\lm_{1}\lm_{4}\bigr) = P_{1312} = -P_{1334} = -P_{2421} = P_{2443},\\
\frac12\lm_{1}\lm_{2} = -\frac12P_{1341} = P_{1413} =-P_{1431} =P_{2324} =-P_{2342}\\
\phantom{\frac12\lm_{1}\lm_{2} }
=-\frac12P_{2432} =P_{3411} =-P_{3422} =P_{3433} =-P_{3444},
\\
\frac12\lm_{3}\lm_{4} = -P_{1211} = P_{1222} =-P_{1233} =P_{1244} =-\frac12P_{1323}\\
\phantom{\frac12\lm_{3}\lm_{4} }
=-P_{1424} =P_{1442} =-P_{2313} =P_{2331} =-\frac12P_{2414},\\
\frac12\lm_{1}\lm_{3} = -P_{1223} = P_{1241} =P_{1421} =P_{1443} =P_{2321}\\
\phantom{\frac12\lm_{1}\lm_{3} }
=P_{2343} =\frac12P_{2422} =\frac12P_{2444} =-P_{3423} = -P_{3441},
\\
\frac12\lm_{2}\lm_{4} = P_{1214} = -P_{1232} =\frac12P_{1311} =\frac12P_{1333} =P_{1412}\\
\phantom{\frac12\lm_{2}\lm_{4} }
=P_{1434} =P_{2312} =P_{2334} =P_{3414} = -P_{3432},\\
\frac12\lm_{1}\lm_{4} = -P_{1213} = P_{1231} = \frac12P_{1321} =P_{2311} =P_{2322} \\
\phantom{\frac12\lm_{1}\lm_{4} }
=P_{2333}=P_{2344} =\frac12P_{2434} =P_{3424} =-P_{3442},
\end{array}
\end{array}
\end{equation}
\begin{equation}
\begin{array}{l}
\begin{array}{l}
\frac12\lm_{2}\lm_{3} = P_{1224} = -P_{1242} = \frac12P_{1343} =P_{1411} =P_{1422} \\
\phantom{\frac12\lm_{2}\lm_{3} }
=P_{1433}=P_{1444} =\frac12P_{2412} =-P_{3413} =P_{3431}.
\end{array}
\end{array}
\end{equation}
\end{subequations}
The rest components are determined by the property $P_{ijk}=-P_{jik}$. Let us remark that $P$ is not a curvature-like tensor.

Obviously, $P=0$ if and only if the corresponding Lie algebra is Abelian and $(L,J,g)$ is a K\"ahler-Norden manifold.

Using \eqref{Phi}, \eqref{Jdim4}, \eqref{g}, \eqref{nabla3},    we get the components $\Phi_{ijk}=\Phi(X_i,X_j,X_k)$ of $\Phi$ and  well as $f_{k}=f(X_k)$ and $f^*_{k}=f^*(X_k)$ of its associated 1-forms. The nonzero of them are the following and the rest are obtained by the property $\Phi_{ijk}=\Phi_{jik}$:
\begin{equation}\label{Phi-ex}
\begin{array}{l}
-\lm_1=-\Phi_{114} = -\Phi_{224} =\Phi_{334} =\Phi_{444}
=\Phi_{132} =\Phi_{242} =\frac14 f_4=-\frac14 f^*_2,
\\
-\lm_2=-\Phi_{113} = -\Phi_{223} =\Phi_{333} =\Phi_{443}
=\Phi_{131} =\Phi_{241} =-\frac14 f_3=\frac14 f^*_1,
\\
-\lm_3=\Phi_{112} = \Phi_{222} =-\Phi_{332} =-\Phi_{442}
=\Phi_{134} =\Phi_{244} =\frac14 f_2=\frac14 f^*_4,
\\
-\lm_4=\Phi_{111} = \Phi_{221} =-\Phi_{331} =-\Phi_{441}
=\Phi_{133} =\Phi_{243} =\frac14 f_1=\frac14 f^*_3.
\end{array}
\end{equation}

The Nijenhuis tensor vanishes on $(L,J,g)$ and $(L,J,\tg)$ as on any $\W_1$-manifold.
According to \cite{GaGrMi}, $N=0$ is equivalent to $\Phi(X_i,X_j)=-\Phi(JX_i,JX_j)$. Then, by means of \eqref{wNPhi} we obtain for the components of the associated Nijenhuis tensor  $S_{ijk}=4\Phi_{ijk}$, where the components of $\Phi$ are given in \eqref{Phi-ex}.

Bearing in mind \eqref{hn=nP}, \eqref{nabla3} and \eqref{Phi-ex}, we get the components of the invariant connection as follows
\begin{equation}\label{hn-ex}
\begin{array}{l}
D_{X_{1}}X_{1} = D_{X_{2}}X_{2} =D_{X_{3}}X_{3} = D_{X_{4}}X_{4} \\
\phantom{D_{X_{1}}X_{1} }=
 -\frac12(\lm_{4}X_{1}+\lm_{3}X_{2}-\lm_{2}X_{3}-\lm_{1}X_{4}),\\
D_{X_{1}}X_{3}=-D_{X_{2}}X_{4}=-D_{X_{3}}X_{1}= D_{X_{4}}X_{2} \\
\phantom{D_{X_{1}}X_{3} }=
 \frac12(\lm_{2}X_{1}-\lm_{1}X_{2}+\lm_{4}X_{3}-\lm_{3}X_{4}),
\\
D_{X_{1}}X_{4} = -D_{X_{3}}X_{2} =
\lm_{1}X_{1} +\lm_{3}X_{3},
\\
D_{X_{2}}X_{3} =-D_{X_{4}}X_{1} =
\lm_{2}X_{2} +\lm_{4}X_{4}.
\end{array}
\end{equation}

After that we compute the basic components $K_{ijk}=K(X_i,X_j)X_k$ of the invariant tensor $K$ under the twin in\-ter\-change, using \eqref{wRbRPhi}, \eqref{bR3} and \eqref{Phi-ex}. In other way, $K_{ijk}$ can be computed directly from \eqref{hn-ex} as the curvature tensor of $D$. We obtain for the components $K_{ijkl}=g(K(X_i,X_j)X_k,X_l)$ the following:
\begin{subequations}
\begin{equation*} \label{wR3}
\begin{array}{l}
\begin{array}{l}
\lm_{1}^{2} = K_{2424},\quad
\lm_{2}^{2} = K_{1313},\quad
\lm_{3}^{2} = K_{2442}, \quad
\lm_{4}^{2} = K_{1331},
\end{array}\\
\begin{array}{l}
\frac12\lm_{1}\lm_{3} = -K_{1223} = K_{1241} =K_{1421} =K_{1443} =K_{2321}\\
\phantom{\frac12\lm_{1}\lm_{3} }
=K_{2343} =\frac12K_{2422} =\frac12K_{2444} =-K_{3423} =K_{3441},
\\
\frac12\lm_{2}\lm_{4} = K_{1214} = -K_{1232} =\frac12K_{1311} =\frac12K_{1333} =K_{1412}\\
\phantom{\frac12\lm_{2}\lm_{4} }
=K_{1434} =K_{2312} =K_{2334} =K_{3414} =-K_{3432},\\
\lm_{1}\lm_{2}+\lm_{3}\lm_{4} = K_{1314} = K_{1332} =K_{2423} =K_{2441},\\
\frac14(\lm_{1}\lm_{2}-\lm_{3}\lm_{4}) = K_{1211} = -K_{1222} =K_{1424} =-K_{1431}\\
\phantom{\frac14(\lm_{1}\lm_{2}-\lm_{3}\lm_{4})}
=K_{2313} =-K_{2342} =K_{3433} =-K_{3444},
\\
\frac14(\lm_{1}\lm_{4}+\lm_{2}\lm_{3}) = -K_{1213} = K_{1224} =K_{1422} =K_{1433}\\
\phantom{\frac14(\lm_{1}\lm_{4}+\lm_{2}\lm_{3})}
=K_{2311} =K_{2344} =K_{3431} =-K_{3442},
\\
\frac14(\lm_{1}\lm_{2}+3\lm_{3}\lm_{4}) = -K_{1233} = K_{1244} =K_{1442} =K_{2331},\\
\frac14(3\lm_{1}\lm_{2}+\lm_{3}\lm_{4}) = K_{1413} = K_{2324} =K_{3411} =-K_{3422},\\
\frac14(\lm_{1}\lm_{4}-3\lm_{2}\lm_{3}) = K_{1242} = -K_{1411} = -K_{1444} =K_{3413},\\
\frac14(3\lm_{1}\lm_{4}-\lm_{2}\lm_{3}) = K_{1231} = K_{2322} =K_{2333} =K_{3424},
\end{array}
\end{array}
\end{equation*}
\begin{equation*} 
\begin{array}{l}
\begin{array}{l}
\end{array}
\\
\begin{array}{ll}
\frac14(\lm_{1}^2-2\lm_{2}^2+\lm_{3}^2) = K_{1432} = K_{3412}, &
\lm_{1}\lm_{2} = -K_{1341} = -K_{2432},\\
\frac14(\lm_{2}^2-2\lm_{3}^2+\lm_{4}^2) = K_{1243} = K_{1423}, &
\lm_{3}\lm_{4} = -K_{1323} = -K_{2414}, \\
\frac14(2\lm_{1}^2-\lm_{2}^2-\lm_{4}^2) = -K_{2341} = K_{3421}, &
\lm_{1}\lm_{4} = K_{1321} = K_{2434},\\
\frac14(\lm_{1}^2+\lm_{3}^2-2\lm_{4}^2) = -K_{1234} = K_{2314}, &
\lm_{2}\lm_{3} = K_{1343} = K_{2412},
\end{array}
\\
\begin{array}{ll}
\frac14(\lm_{1}^2+2\lm_{2}^2+\lm_{3}^2) = -K_{1212},\quad &
\frac14(2\lm_{1}^2+\lm_{2}^2+\lm_{4}^2) = K_{1221},\\
\frac14(\lm_{1}^2+\lm_{3}^2+2\lm_{4}^2) = K_{3434},\quad &
\frac14(\lm_{2}^2+2\lm_{3}^2+\lm_{4}^2) = -K_{3443},\\
\frac14(3\lm_{1}^2+3\lm_{3}^2-2\lm_{4}^2) = K_{1414},\quad &
\frac14(2\lm_{1}^2-3\lm_{2}^2-3\lm_{4}^2) = -K_{1441},\\
\frac14(3\lm_{2}^2-2\lm_{3}^2+3\lm_{4}^2) = K_{2323},\quad &
\frac14(3\lm_{1}^2-2\lm_{2}^2+3\lm_{3}^2) = K_{2332}.
\end{array}
\end{array}
\end{equation*}
\end{subequations}
The rest components are determined by property $K_{ijk}=-K_{jik}$. Let us remark that $K$ is not a curvature-like tensor.

Obviously, $K=0$ if and only if the corresponding Lie algebra is Abelian and $(L,J,g)$ is a K\"ahler-Norden manifold.


\end{document}